\documentclass[12pt,reqno]{amsart}

\usepackage{epsfig}
\usepackage{amsmath, amsthm, amsfonts}
\usepackage{graphicx}

\numberwithin{equation}{section}

\newcommand{\lltwos}{\ell^2({\mathbb S})}
\newcommand{\wlltwon}{\ell_{a}^2({\mathbb N})}
\newcommand{\wlltwonp}{\ell_{a'}^2({\mathbb N})}
\newcommand{\wlltwoz}{\ell_{a}^2({\mathbb Z})}
\newcommand{\wlltwozp}{\ell_{a'}^2({\mathbb Z})}
\newcommand{\wlltwos}{\ell_{a}^2({\mathbb S})}
\newcommand{\wlltwosp}{\ell_{a'}^2({\mathbb S})}

\newcommand{\wlltwosc}{\ell_{a^{(n)}}^2({\mathbb S})}
\newcommand{\di}{\displaystyle}

\newcommand{\D}{{\mathbb D}}
\newcommand{\A}{{\mathbb A}}

\newcommand{\C}{{\mathbb C}}
\newcommand{\N}{{\mathbb N}}

\newcommand{\Z}{{\mathbb Z}}

\renewcommand{\d}{\partial}

\newcommand{\f}{\varphi}

\newtheorem{theo}{{\sc \bf Theorem}}[section]
\newtheorem{cor}[theo]{{\sc \bf Corollary}}
\newtheorem{lem}[theo]{{\sc \bf Lemma}}
\newtheorem{prop}[theo]{{\sc \bf Proposition}}

\newenvironment{defin}{\medskip\noindent{\it Definition:\/} }{\medskip}

\begin{document}

\title{D-bar Operators on Quantum Domains}

\author{Slawomir Klimek}
\address{Department of Mathematical Sciences,
Indiana University-Purdue University Indianapolis,
402 N. Blackford St., Indianapolis, IN 46202, U.S.A.}
\email{sklimek@math.iupui.edu}

\author{Matt McBride}
\address{Department of Mathematical Sciences,
Indiana University-Purdue University Indianapolis,
402 N. Blackford St., Indianapolis, IN 46202, U.S.A.}
\email{mmcbride@math.iupui.edu}

\thanks{}

\date{December 30, 2009}

\begin{abstract}
We study the index problem for the d-bar operators subject to Atiyah- 
Patodi-Singer boundary conditions on noncommutative disk and annulus.
\end{abstract}

\maketitle
\section{Introduction}

It this paper we consider noncommutative analogs of the d-bar operator  
on simple complex plane domains with boundary: disk and annulus. In both cases
the corresponding quantum domain, its boundary, a d-bar operator, and an analog of the $L^2$ Hilbert space of functions on the domain is constructed using a weighted shift, subject to suitable assumptions. The weighted shift plays the role of the complex coordinate $z$. 

For such d-bar operators we consider boundary conditions of Atiyah,  
Patodi, Singer (APS) type \cite{APS}. This can be done so that both the commutative and the noncommutative setup appear in close analogy. The main result of the paper is that of the quantum d-bar operators subject
to APS conditions are unbounded Fredholm operators. Additionally we  
compute their index.

Let us recall that an unbounded operator $D$ is called a Fredholm operator 
if $D$ is closed, has closed range, and finite dimensional kernel and cokernel. Equivalently, see \cite{S}, a closed opearotor $D$ is Fredholm if it has a bounded parametrix $Q$ such that both $QD-I$ and $DQ-I$ are compact. The technical part of the paper consist of finding such a parametrix.

The celebrated APS boundary condition was introduced in \cite{APS} to handle the index theory for geometrical operators on manifolds with boundary when usual local boundary conditions were not available. Because it is non-local, the APS condition seems to be naturally suited to consider in noncommutative geometry. A more general class of APS-type boundary conditions was described in \cite{BBW}. Here we consider only simple APS-type boundary conditions given by spectral projections.

This paper is a continuation and an extension of \cite{CKW}, which  
considered APS theory on the noncommutative unit disk. Here we present somewhat different and more detailed treatment of the disk case as well as a similar theory on the cylinder. In particular the modifications we consider here yield a compact parametrix for the d-bar operators, which was not the case in \cite{CKW}. The present paper will be followed by a separate note containing a construction of a parametrix for the quantum d-bar operator on the semi-infinite cylinder i.e. a punctured disk.

Noncommutative domains considered in this paper were previously discussed in  \cite{KL1, KL3}. Other papers that studied d-bar operator in similar situations (but not the APS boundary conditions) are: \cite{BKLR}, \cite{K}, \cite{RW}, \cite{SSV1}-\cite{SSV5}. A related study of an example of APS boundary conditions in the context of noncommutative geometry is contained in \cite{CPR}, another one is in \cite{L}. 

The ideas in this paper can be further extended in several directions. The present setup fits into deformation-quantization scheme and so it will be desirable to consider classical limit of the quantum d-bar operators. Other, different, possibly higher dimensional examples should also be constructed. Because of the compact parametrix, the d-bar operators of this paper can be used to define Fredholm modules over quantum domains (with boundary), which will be interesting to explore. While the computation of the index in the present work is fairly straightforward, it is a challenging question to find a noncommutative framework for such calculations in general.

The paper is organized as follows.
In the preliminary Section 2 we describe the classical d-bar operators  
on domains in complex plane subject to APS-type boundary conditions and compute their index.
Section 3 contains the main constructions of the paper: quantum disk, quantum annulus, Hilbert spaces, d-bar operators, APS-type boundary conditions. The  
main results are also stated in this section.
Section 4 is the longest of the paper. It contains detailed analysis  
of some finite difference operators in weighted $\ell^2$ spaces. The operators are essentially unbounded Jacobi operators, see \cite{T}. That analysis constitutes the technical backbone of the paper.
Section 5 introduces noncommutative Fourier transform on our quantum  
domains. The Fourier transform essentially diagonalizes the d-bar operators and thus reduces their analysis to the analysis of the difference operators of
the previous section. Finally, Section 6 describes proofs of the main  
results.

\section{The d-bar operator on domains in the complex plane}

It this section we review the basic aspects of the APS theory for the d-bar operator on simple domains in the complex plane $\C$. 
We start by introducing some notation. The first domain is the  disk:

\begin{equation*}
\begin{aligned}
&\D = \di\{z\in \C : \ |z| \le \rho_+\} \\
&\d\D = \di\{z\in \C : \ |z| = \rho_+\}\simeq S^1 .
\end{aligned}
\end{equation*}
The second domain is an annulus in the complex plane $\C$:

\begin{equation*}
\begin{aligned}
&\A_{\rho_{-},\rho_{+}} = \di\{z\in \C : \ 0<\rho_{-} \le |z| \le \rho_{+} \} \\
&\d\A_{\rho_{-},\rho_{+}} = \di\{z\in \C : \ |z| = \rho_{\pm} \} \simeq S^1 \cup S^1 ,
\end{aligned}
\end{equation*}
which can also be viewed as a finite cylinder. 

For each of those domains we will consider the d-bar operator:

\begin{equation*}
D = \frac{\d}{\d\overline{z}}
\end{equation*}
defined on the space of smooth functions. 

First we will concentrate on the unit disk. In this case we have the short exact sequence:

\begin{equation}\label{seq1}
0\longrightarrow C_0^\infty(\D) \longrightarrow C^\infty(\D) \overset{r}{\longrightarrow} C^\infty(\d\D) \longrightarrow 0
\end{equation}
where $r: \ C^\infty(\D) \to C^\infty(\d\D)$ is the restriction map to the boundary, 
$rf(\varphi) = f(1\cdot e^{i\varphi})$.
Here $C_0^\infty(\D)$ is the space of smooth functions on $\D$ vanishing at the boundary and $z\in\D$ has polar representation $z = \rho e^{i\varphi}$.   

Now we consider the APS-like boundary conditions on $D$. Notice that the APS theory cannot be applied directly in this case since the operator $D$ does not quite decompose into tangential (boundary) and transverse parts near boundary. However this is only a minor technical annoyance, and it is clear that $-i\d/\d\varphi$ is the correct boundary operator. The APS-type boundary conditions considered in this paper are given in terms of the spectral projections of the boundary operator 
$-i\d/\d\varphi$ as follows. Let $\pi_A\,(I)$ be the spectral projection of a self-adjoint operator $A$ onto interval $I$.
For an integer $N$ we introduce $P_N$:
\begin{equation}\label{proj1}
P_N = \pi_{\frac{1}{i}\frac{\d}{\d\f}}(-\infty,N] .
\end{equation}
In other words $P_N$ is the orthogonal projection in $L^2(S^1)$ onto 
span$\{e^{in\varphi}\}_{n\le N}$.   

The main object of the APS theory is the operator $D_N$ defined to be the operator $D$ with the domain:

\begin{equation*}
\textrm{dom}(D_N) = \{f\in C^\infty(\D)\subset L^2(\D) \ : \ rf\in \textrm{Ran } P_N \}.
\end{equation*}

We have the following theorem, see \cite{CKW} for details.

\begin{theo}
The closure of the operator $D_N$ is an unbounded Fredholm operator in $L^2(\D)$ and it has the following index: $Index(D_N) = N+1.$

\end{theo}

Now we will discuss the annulus. While we skip some functional analytic details, we show the index calculation in a similar fashion to what was done in 
\cite{CKW} in the disk case.

If one lets $r_\pm$ be the restriction to the boundary map i.e. $r_\pm f(\varphi) = f(\rho_\pm e^{i\varphi})$, then one has the short exact sequence:

\begin{equation}\label{seq2}
0\longrightarrow C_0^{\infty}(\A_{\rho_{-},\rho_{+}}) \longrightarrow C^{\infty}(\A_{\rho_{-},\rho_{+}}) \overset{r=r_+ \oplus r_-}{\longrightarrow} C^{\infty}(S^1) \oplus C^{\infty}(S^1) \longrightarrow 0
\end{equation}
where $C_0^{\infty}(\A_{\rho_{-},\rho_{+}})$ is the space of smooth functions on $\A_{\rho_{-},\rho_{+}}$ which are zero on the boundary.
 
The key to index calculation of the d-bar operator is the following proposition. In what follows we use the usual inner product on $L^2(\A_{\rho_{-},\rho_{+}})$:

\begin{equation*}
\di\langle f, g \di\rangle = \int_{\A_{\rho_{-},\rho_{+}}}\overline{f(z)}g(z)\frac{dz\wedge d\overline{z}}{-2i\pi}.
\end{equation*}

\begin{prop}\label{intro2}
Let $D$ be the operator

\begin{equation*}
D = \frac{\d}{\d\overline{z}}
\end{equation*}
on $C^\infty(\A_{\rho_{-},\rho_{+}})$.   Then the kernel of $D$ is
the set of bounded holomorphic functions on $\A_{\rho_{-},\rho_{+}}$. Moreover

\begin{equation*}
\di\langle Df, g \di\rangle = \di\langle f, \overline{D}g \di\rangle + \int_0^{2\pi}\overline{r_+f(\f)}r_+g(\f)\rho_+e^{-i\f}\frac{d\f}{2\pi} - \int_0^{2\pi}\overline{r_-f(\f)}r_-g(\f)\rho_-e^{-i\f}\frac{d\f}{2\pi}
\end{equation*}
where $f,g\in C^\infty(\A_{\rho_{-},\rho_{+}})$ and

\begin{equation*}
\overline{D} = -\frac{\d}{\d z} .
\end{equation*}
\end{prop}

\begin{proof}
The first conclusion is clear. The integration by parts formula follows immediately from Stokes' Theorem. \end{proof}

In order to define APS-type boundary conditions here we take extra caution since the boundary has two components. Let $P_N^\pm$ be the spectral projections in $L^2(S^1)$ of the boundary operators
${\pm\frac{1}{i}\frac{\d}{\d\f}}$ onto interval $(-\infty,N]$ i.e.:

\begin{equation}\label{proj2}
P_N^{\pm} = \pi_{\pm\frac{1}{i}\frac{\d}{\d\f}}\,(-\infty,N]
\end{equation}
where $\pm$ is introduced due to the boundary orientations of the inner circle and outer circle.   Then, for integers $M$, $N$, we define the operator 
$D_{M,N}$ to be equal to $D$ with domain

\begin{equation*}
\textrm{dom}(D_{M,N}) = \{ f\in C^{\infty}(\A_{\rho_{-},\rho_{+}}) \ : \ r_+ f \in \textrm{Ran} \ P_M^+ , \ r_- f \in \textrm{Ran} \ P_N^- \}.
\end{equation*}
An immediate corollary of this definition is the description of the kernel of $D_{M,N}$.

\begin{cor}
Let $D_{M,N}$ be as defined above, then

\begin{equation*}
\textrm{Ker}(D_{M,N}) = \left\{
\begin{array}{cc}
\left\{f \ : \ f(z) = \sum_{n=-N}^M c_n z^n\right\} & \textrm{if}\ N+M\ge 0 \\
0 & \textrm{otherwise.}
\end{array}\right.
\end{equation*}
\end{cor}

It follows from proposition (\ref{intro2}) that the adjoint of $D_{M,N}$, is (the closure of) the operator $\overline{D}_{M,N}$ which is equal to $\overline{D}$ but with the following domain

\begin{equation*}
\textrm{dom}(\overline{D}_{M,N}) = \{ f\in C^{\infty}(\A_{\rho_{-},\rho_{+}}) \ : \ e^{-i\f}r_+ f \in \textrm{Ker} \ P_M^+, \ e^{-i\f}  r_-f \in \textrm{Ker} \ P_N^- \} .
\end{equation*}
Moreover, one has the following description of the kernel of $\overline{D}_{M,N}$

\begin{equation*}
\textrm{Ker}(D_{M,N}) = \left\{
\begin{array}{cc}
\left\{f \ : \ f(z) = \sum_{n=N}^{-(M+2)} c_n z^n\right\} &\textrm{if}\ N+M< 0 \\
0 & \textrm{otherwise.}
\end{array}\right.
\end{equation*}

The following theorem is the corresponding index theorem for the commutative cylinder.

\begin{theo}\label{intro3}
The closure of the operator $D_{M,N}$ is an unbounded Fredholm operator.  Its index is given by:
$Index(D_{M,N}) = M+N+1$.

\end{theo}

\begin{proof}
To show the Fredholm property one follows \cite{APS}.  If $f\in C^\infty(\A_{\rho_{-},\rho_{+}})$ then $f(z)$ has the following Fourier representation:

\begin{equation*}
f(z) = \di\sum_{n\in\Z} f_n(\rho)e^{in\varphi}  .
\end{equation*}
This Fourier representation is exactly the spectral decomposition of \cite{APS} using the eigenvectors of the boundary operators $\pm i\d/\d\varphi$. In the Fourier transform the operator $D$ decomposes into sum of ordinary differential operators which allows for explicit calculation of a parametrix just like in \cite{APS}.

The index computation is as follows.  
We have:

\begin{equation*}
\begin{aligned}
\textrm{dim Ker}(D_{M,N}) &= \#\{n \ | \ -N\le n\le M \} \\
&=\left\{\begin{array}{cc}
0 & \textrm{if} \ M+N <0 \\
M+N+1 & \textrm{if} \ M+N\ge 0.
\end{array}\right.
\end{aligned}
\end{equation*}
In a similar fashion

\begin{equation*}
\begin{aligned}
\textrm{dim Ker}(D_{M,N}^*) &= \#\{n \ | \ N\le n\le -(M+2) \} \\
&=\left\{\begin{array}{cc}
-(M+N+1) & \textrm{if} \ N <0 \\
0 & \textrm{if} \ N\ge 0.
\end{array}\right.
\end{aligned}
\end{equation*}
Consequently

\begin{equation*}
\textrm{Index}(D_{M,N}) = \textrm{dim Ker}(D_{M,N}) - \textrm{dim Ker}(D_{M,N}^*)=M+N+1.
\end{equation*}

\end{proof}

We now turn our attention to the d-bar operator in the quantum domains.

\section{The d-bar operator on the non-commutative domains}

In this section we define the main objects of this paper: quantum disk, quantum annulus, Hilbert spaces of $L^2$ ``functions", and d-bar operators. The  
main results are also stated at the end of this section. 

In the following definitions we let $\mathbb{S}$ be either $\N$ or $\Z$. The main input of the theory is a weighted shift $U_W$ in $\lltwos$. Conceptually,
$U_W$ is a noncommutative complex coordinate on the corresponding noncommutative domain.

\begin{defin}
Let $\{e_k\}$, $k\in \mathbb{S}$ be the canonical basis for $\lltwos$. Given a 
bounded sequence of numbers $\{w_k\}$, called weights, the weighted shift $U_W$ is an operator in $\lltwos$ defined by:

\begin{equation*}
U_W e_k = w_ke_{k+1}.
\end{equation*}
\end{defin}

We will also need the usual shift operator $U$ which is defined by

\begin{equation*}
Ue_k = e_{k+1}
\end{equation*}
and the diagonal operator $W$ defined by

\begin{equation}\label{Wdefref}
W e_k = w_ke_k.
\end{equation}
Note that $U_W$ decomposes to $U_W=UW$ and $W=(U_W^*U_W)^{1/2}$ as in the polar decomposition.  If $\mathbb{S}=\N$ then the shift $U_W$ is called unilateral and it will be used to define a quantum disk. If $\mathbb{S}=\Z$ then the shift $U_W$ is called bilateral and it will be used to define a quantum annulus (also called a quantum cylinder).

We require the following conditions on $U_W$:

\medskip

{\it Condition 1.\ } The weights are uniformly positive $w_k\geq \epsilon> 0$, for every $k\in \mathbb{S}$.
\medskip

{\it Condition 2.\ } The shift $U_W$ is hyponormal, i.e.

\begin{equation*}
S = \left[U_W^*,U_W\right]\ge 0.
\end{equation*}

\medskip

{\it Condition 3.\ } The operator $S$ defined in {\it condition 2} is injective.

\medskip

Let us remark on some implications of these conditions. First note how $S$ acts on the basis $\{e_k\}$

\begin{equation}\label{Sdefref}
\begin{aligned}
Se_k &= (U_W^*U_W - U_WU_W^*)\,e_k \\
&= (w_k^2 - w_{k-1}^2)\,e_k=s_ke_k,
\end{aligned}
\end{equation}
where $s_k := w_k^2 - w_{k-1}^2$.  It follows that the conditions 2 and 3 mean that the weights $w_k$ form a strictly increasing sequence.  Hence the following limits exist and are positive numbers:
\begin{equation*}
w^\pm:=\lim_{k\to\pm\infty}w_k.
\end{equation*}

Secondly, observe that $S$ is a trace class operator with easily computable trace: $\textrm{tr}(S)=w^+$ in the unilateral case and 
$\textrm{tr}(S)=w^+-w^-$ in the bilateral case. Moreover $S$ is invertible with unbounded inverse.






Let $C^*(W)$ be the $C^*-\textrm{algebra}$ generated by $U_W$.   Then it is known that there are short exact sequences analogous to \ref{seq1} and \ref{seq2}.   Let $\mathcal{K}$ be the ideal of compact operators.  Then in the unilateral case the $C^*-\textrm{algebra}$ generated by $U_W$ is the Non-Commutative Disk of \cite{KL1} with the following short exact sequence:

\begin{equation*}
0\longrightarrow \mathcal{K} \longrightarrow C^*(W) \overset{r}{\longrightarrow} C(S^1) \longrightarrow 0.
\end{equation*}

Similarly, in the bilateral case the $C^*-\textrm{algebra}$ generated by $U_W$  is the Non-Commutative Cylinder, see \cite{KL3}, with the following short exact sequence:

\begin{equation*}
0\longrightarrow \mathcal{K} \longrightarrow C^*(W) \overset{r=r_+\oplus r_-}{\longrightarrow} C(S^1) \oplus C(S^1) \longrightarrow 0.
\end{equation*}
In the above we let again, abusing notation, $r$ be the restriction map in the disk case and $r_\pm$ in the cylinder case. These two sequences are described in \cite{Conway2}.

Now we proceed to the definitions of the quantum d-bar operators. With slight abuse, we will use the same notation for both classical and quantum operators.

We define the Hilbert space $\mathcal{H}$ as the completion of $C^*(W)$ with respect to the inner product 
$\langle \ , \ \rangle_S$  defined as follows:

\begin{equation*}
\langle a,b\rangle_S = \textrm{tr}(S^{1/2}bS^{1/2}a^*)
\end{equation*}
where $a,b\in C^*(W)$. It is easy to verify that $\langle a,a\rangle_S$ is well-defined and positive.
Note that the inner product $\langle \ , \ \rangle_S$ is slightly different than the one defined in \cite{CKW}. This is done (among other reasons) to make definitions more symmetric. 

Next  we define a quantum d-bar  operator
$D$ in $\mathcal{H}$ by the following expression:

\begin{equation*}
Da = S^{-1/2}\left[a,U_W\right]S^{-1/2}
\end{equation*}
where the domain of $D$ is the set of those $a\in\mathcal{H}$ for which $S^{1/2}DaS^{1/2}(Da)^*$ is trace class.  It will be verified later that
$\textrm{Dom}(D)$ is dense and that for $a\in \textrm{Dom}(D)$, $r(a)$ is a square integrable function on the boundary of the domain.
This definition is again somewhat different than the one considered in \cite{CKW}: it is symmetric with respect to left/right multiplication, and the operator $D$ has better functional-analytic properties.

A straightforward computation shows the following identities:

\begin{equation*}
\begin{aligned}
D(U_W^n) &= 0 \\
D(U_W^*) &= 1 \\
D((U_W^*)^n) & = S^{-1/2}\left[(U_W^*)^n,U_W\right] S^{-1/2}=\\
&= S^{-1/2}(U_W^*)^{n-1}S^{1/2} - S^{-1/2}(U_W^*)^{n-2}SU_W^*S^{-1/2} -  \cdots - S^{1/2}(U_W^*)^{n-1}S^{-1/2}.
\end{aligned}
\end{equation*}


The first two computations  show that $D$ looks like $\frac{\d}{\d\overline{z}}$ if $U_W$ was $z$ and the third computation illustrates the non-commutativity of the situation. 

We proceed to the definitions of the APS-type boundary conditions on $D$. Let again $P_N$ be the orthogonal projection in $L^2(S^1)$ defined in \ref{proj1}, and let $P_N^{\pm}$ be the orthogonal projections  defined in \ref{proj2}. 
Now we can define $D_N$, $D_{M,N}$  in full analogy with the previous section.  The operator $D_N$ equals the unilateral operator $D$ with domain

\begin{equation*}
\textrm{dom}(D_N) = \left\{ a\in \textrm{Dom}(D) : \ r(a)\in \textrm{Ran} \ P_N \right\} .
\end{equation*}
Similarly, the operator $D_{M,N}$ equals the bilateral operator $D$ with domain

\begin{equation*}
\textrm{dom}(D_{M,N}) = \left\{ a\in \textrm{Dom}(D) : \ r_+(a)\in \textrm{Ran} \ P_N^+ ,\ r_-(a)\in\textrm{Ran} \ P_M^-\right\} .
\end{equation*}

We are now in a position to state the main results of this paper.

\begin{theo}
For the non-commutative disk case,  the operator $D_N$ is an unbounded Fredholm operator. Moreover $ind(D_N) = N+1$.  
\end{theo}

This is a slight modification from \cite{CKW}, where a somewhat different version of $D_N$ was considered. We additionally have:

\begin{theo}
For the non-commutative cylinder case,  the operator $D_{M,N}$ is an unbounded Fredholm operator.   
Moreover $ind(D_{M,N}) = M+N+1$.
\end{theo}
The proofs are contained in the last section.

\section{Analysis of finite difference operators}

In this section we present a detailed analysis of certain finite difference operators related to Jacobi matrices.  As indicated in the introduction, these operators come up us components of $D$ and its adjoint in Fourier transforms. This will be fully explained in the following section.

As before $\mathbb{S}$ is either $\Z$ or $\N$. Given a sequences of positive numbers $a = \{a_n\}_{n\in\mathbb{S}}$ called weights, the Hilbert Space $\wlltwos$ is defined by

\begin{equation*}
\wlltwos = \left\{f=\{f_n\}_{n\in\mathbb{S}} \ : \ \sum_{n\in\mathbb{S}} \frac{1}{a_n}|f_n|^2 < \infty \right\}
\end{equation*}
with inner product given by $\di\langle f, g \di\rangle = \sum_{n\in\mathbb{S}} \frac{1}{a_n}\overline{f_n}g_n$.
If a sequence $\{f_n\}\in\wlltwos$ has limits, $\lim\limits_{n\to\pm\infty} f_n$, they will be denoted $f_{\pm\infty}$.

Given two weight sequences $a$ and $a'$ we will be studying throughout this section the following unbounded Jacobi type difference operators between $\wlltwos$ and $\wlltwosp$:

\begin{equation*}
\begin{aligned}
Af_n &= a_n(f_n - c_{n-1}f_{n-1}) \quad \textrm{where} \\
\textrm{dom}(A) &= \left\{f\in\wlltwosp \ : \ \|Af\|_{\wlltwos}<\infty \right\} \\
&\quad\textrm{and} \\
\overline{A}f_n &= a_n'(f_n - \overline{c_n}f_{n+1}) \quad \textrm{where} \\
\textrm{dom}(\overline{A}) &= \left\{f\in\wlltwos \ : \ \|\overline{A}f\|_{\wlltwosp}<\infty \right\}
\end{aligned}
\end{equation*}
for $n\in  \mathbb{S}$. If $\mathbb{S}=\N$ we assume in the above that $f_{-1}=0$. 

The coefficients  $a_n$, $a'_n$, and 
$c_n\in \C$ are assumed to satisfy: 

\begin{equation}\label{acconditions}
0<|c_n|\le 1\ ,\ \sum_{n\in\mathbb{S}} \frac{1}{a_n'} =C'<\infty \ , \ \sum_{n\in\mathbb{S}} \frac{1}{a_n}=C <\infty \ , \prod_{n\in\mathbb{S}} \frac{1}{c_n} < \infty .
\end{equation}
We also define:

\begin{equation*}
K = \prod_{n\in\mathbb{S}} \frac{1}{|c_n|}.
\end{equation*}

The goal of this section  is to establish the Fredholm properties of the operators $A$, 
$\overline{A}$ and related operators obtained by imposing conditions at infinities. 
This is done by constructing a parametrix for each operator.
Our discussion will be split into two separate but similar cases: unilateral and bilateral.

\subsection{Unilateral Case}

We  first  study the kernels of $A$ and $\overline{A}$, in order to see if these operators have inverses or not.

\begin{prop}\label{prelimsker}
Given $A$ and $\overline{A}$ above we have

\begin{equation*}
\begin{aligned}
\textrm{Ker}\,{A} &= \{0\} \\
\textrm{dim} \ \textrm{Ker}\,\overline{A} & = 1.
\end{aligned}
\end{equation*}

\end{prop}

\begin{proof}
First consider the equation $Af_n = 0$ which is $a_n(f_n - c_{n-1}f_{n-1}) = 0$ for $n=0,1,2\ldots$  Then solving recursively  one can see that the only solution to the equation is $f_0 = f_1 = \cdots = f_n = 0$ for all $n$.  This shows that $\textrm{Ker }A$ is trivial and thus $A$ is an invertible operator.

Secondly consider the equation $\overline{A}f_n = 0$ which is $a_n'(f_n-\overline{c_n}f_{n+1})=0$ for $n=0,1,2\ldots$   Then solving recursively one has

\begin{equation*}
\begin{array}{cc}
n=0 \Rightarrow & f_1 = \frac{1}{\overline{c_0}}f_0 \\
n=1 \Rightarrow & f_2 = \frac{1}{\overline{c_0c_1}}f_0 \\
\vdots & \vdots
\end{array}
\end{equation*}
which in general gives

\begin{equation*}
f_n = \frac{1}{\overline{c_0c_1\cdots c_{n-1}}}f_0,
\end{equation*}
thus showing that $\overline{A}$ has a one dimensional kernel provided that $f_n\in\wlltwon$.  Notice the following

\begin{equation*}
|f_n| = \frac{1}{|c_0\cdots c_{n-1}|}|f_0|\le \prod_{i=0}^\infty \frac{1}{|c_i|}|f_0|
=K|f_0|
\end{equation*}
since $|c_i| \le 1$ for all $i=0,1,\ldots$   From this it follows that

\begin{equation*}
\|f\|^2 \le \sum_{n=0}^\infty \frac{1}{a_n}K^2|f_0|^2 = CK^2|f_0|^2 <\infty
\end{equation*}
with the constants defined at the beginning of the section.
Thus this completes the proof.
\end{proof}

Next we show how to find the inverse $T$ of $A$ and we study its properties.

\begin{prop}\label{inverseun}
There exists an operator $T \in B(\wlltwon,\wlltwonp)$ such that $TA=I_{\wlltwon}$ and $AT = I_{\wlltwonp}$. Indeed it is given by the formula \ref{Tdef} below. In particular  $A$ is an unbounded Fredholm operator with zero index.
\end{prop}

\begin{proof}
From proposition ($\ref{prelimsker}$) we know that $A$ is invertible so let $\{g_n\} \in \wlltwon$ and $\{f_n\} \in \textrm{dom}(A)$ and consider the equation $Af_n = g_n$ which is $a_n(f_n - c_{n-1}f_{n-1}) = g_n$ for $n=0,1,2\ldots$  As above, solving for each $n$ recursively one arrives at the following formula

\begin{equation}\label{Tdef}
Tg_n = \di\sum_{i=0}^n \di\frac{1}{a_i} \di\left(\di\prod_{j=i}^{n-1}c_j\right)g_i ,
\end{equation}
where in the above we set, for convenience:

\begin{equation*}
\prod_{j=n}^{n-1}c_j = 1 .
\end{equation*}

Next we show that $T \in B(\wlltwon, \wlltwonp)$.   We divide and multiply each term as follows

\begin{equation*}
Tg_n = \frac{1}{a_n}g_n + \frac{c_{n-1}}{a_{n-1}}g_{n-1} + \cdots + \frac{c_{n-1}\cdots c_0}{a_0}g_0 =
\end{equation*}

\begin{equation*}
=\frac{\sqrt{a_n}}{a_n}\frac{g_n}{\sqrt{a_n}} + \frac{c_{n-1}\sqrt{a_{n-1}}}{a_{n-1}}\frac{g_{n-1}}{\sqrt{a_{n-1}}} + \cdots + \frac{c_{n-1}\cdots c_0\sqrt{a_0}}{a_0}\frac{g_0}{\sqrt{a_0}}.
\end{equation*}
Since $\|Tg\|^2 = \di\sum_{n=0}^\infty\frac{1}{a_n} |Tg_n|^2$ and since $|c_n| \le 1$ for every $n$, using the Cauchy - Schwartz inequality one has

\begin{equation*}
\begin{aligned}
\di\left|Tg_n\right|^2 &\le \di\left(\di\left(\frac{\sqrt{a_n}}{a_n}\right)^2 + \cdots + \di\left(\frac{\sqrt{a_0}}{a_0}\right)^2\right)\di\left(\frac{1}{a_n}|g_n|^2 + \cdots + \frac{1}{a_0}|g_0|^2\right) \le \\
&\le \di\left(\di\sum_{n=0}^\infty \frac{1}{a_n}\right)\|g\|^2 = C\|g\|^2.
\end{aligned}
\end{equation*}
Consequently:

\begin{equation*}
\begin{aligned}
\|Tg\|^2 &\le \sum_{n=0}^\infty\frac{1}{a_n'}C\|g\|^2 =\\
&= C'C\|g\|^2,
\end{aligned}
\end{equation*}
which implies that $\|T\| \le \sqrt{C'C}$, thus one has $T \in B(\wlltwon, \wlltwonp)$.
A straightforward calculation shows that $TA=I_{\wlltwon}$ and $AT = I_{\wlltwonp}$.
\end{proof}

An important corollary from this proposition is the existence of limits at infinity for sequences which are in the domain of $A$.

\begin{cor}\label{limituni}
Let $f=\{f_n\} \in\textrm{dom}(A)$, then $\lim\limits_{n\to\infty} f_n = f_\infty$ exists and is given by the following formula

\begin{equation*}
f_\infty =\sum_{i=0}^\infty \frac{1}{a_i}\left(\prod_{j=i}^\infty c_j\right)Af_i .
\end{equation*}

\end{cor}

\begin{proof}
If $f\in\textrm{dom}(A)$, then write $f$ as, $f=T(Af)$, then one has the following

\begin{equation*}
f_n = \sum_{i=0}^n \frac{1}{a_i}\left(\prod_{j=i}^{n-1}c_j\right)Af_i .
\end{equation*}
Taking $n\to\infty$ gives the formula above.
\end{proof}

We now wish to consider the operator $\overline{A}$ and determine if it has bounded right inverse since proposition ($\ref{prelimsker}$) tells us that $\overline{A}$ has a one dimensional kernel. The next proposition will show this.
We will be using the following notation: if $V$ be a closed subspace of a Hilbert space $H$, then we denote $\textrm{Proj}_{V}$, to be the orthogonal projection onto $V$.  

\begin{prop}\label{inverseunker}
Given $\overline{A}$ from above then there exists a $\overline{T} \in B(\wlltwonp,\wlltwon)$ such that $\overline{AT} = I_{\wlltwonp}$ and $\overline{TA} = I_{\wlltwon} - \textrm{Proj}_{\textrm{Ker }\overline{A}}$. In particular  $\overline{A}$ is an unbounded Fredholm operator
with index equal to one.
\end{prop}

\begin{proof}
From proposition ($\ref{prelimsker}$) we know that $\overline{A}$ has a one dimensional kernel spanned by the following vector $\Omega \in \textrm{Ker}(\overline{A})$:

\begin{equation*}
\Omega_n=\prod_{i=n}^\infty \overline{c_i}=\left(\prod_{i=0}^{n-1}\frac{1}{\overline{c_i}}\right)
\left(\prod_{i=0}^\infty \overline{c_i}\right)  .
\end{equation*}

Next consider the equation $\overline{A}g_n = a_n'(g_n - \overline{c_n}g_{n+1}) = f_n$ for $n=0,1,2,\ldots$   As before solve the equation recursively and one will arrive at the formula

\begin{equation*}
g_{n+1} =  \di\prod_{i=0}^n \frac{1}{\overline{c_i}}g_0 - \di\sum_{i=0}^n \frac{1}{a_i'}\di\left(\di\prod_{j=i}^n \frac{1}{\overline{c_j}}\right)f_i .
\end{equation*}
where $g_0$ is arbitrary.   To finish the construction of $\overline{T}$ we need to choose $g_0$ so that  
$\overline{TA} = I_{\wlltwoz} - \textrm{Proj}_{\textrm{Ker }\overline{A}}$ as it's clear that $\overline{AT} = I_{\wlltwozp}$.   

The disadvantage of the above formula for $\overline{T}$ is that it does not translate easily to the bilateral case. Anticipating it, we rewrite the above solution in an equivalent but different looking form:
\begin{equation}\label{tbardef}
\begin{aligned}
\overline{T}f_n&=g_n = \sum_{i=n}^\infty \frac{1}{a_i'}\left(\prod_{j=n}^{i-1}\overline{c_j}\right)f_i -\left(\prod_{i=n}^\infty \overline{c_i}\right)L(f)\\
&=\overline{T_0}f_n-\Omega_n\,L(f),\\
\end{aligned}
\end{equation}
where we set $\prod_{j=n}^{n-1}\overline{c_j} = 1$  and $L(f)$ is an arbitrary constant. This form of solution is also explained conceptually when considering bilateral case.

For $\overline{T}f$ to be orthogonal to $\textrm{Ker }\overline{A}$, one needs $\langle \Omega, \overline{T}f\rangle = 0$ for the above $\Omega \in \textrm{Ker }\overline{A}$.   From this one can deduce that $L(f)$ is the following following linear functional of $f$:

\begin{equation*}
L(f) := \frac{\langle \Omega, \overline{T_0}f\rangle}{||\Omega||^2}=
\frac{\sum_{n=0}^\infty\sum_{i=n}^\infty\frac{1}{a_n'}\frac{1}{a_i'}\left(\prod_{j=n}^{i-1}c_j\right)\left(\prod_{k=n}^\infty \overline{c_k}\right)f_i}{\sum_{n=0}^\infty\frac{1}{a_n'}\left(\prod_{i=n}^\infty|c_i|^2\right)}.
\end{equation*}

It is straightforward to verify now that  $\overline{TA} = I_{\wlltwoz} - \textrm{Proj}_{\textrm{Ker }\overline{A}}$ and that $\overline{AT} = I_{\wlltwozp}$.
All that remains is to show the boundedness of $\overline{T}$. The operator
$\overline{T_0}$ is bounded by $\sqrt{CC'}$ in exactly the same way as the operator $T$ is proposition (\ref{inverseun}). To estimate $L(f)$ we notice that
\begin{equation*}
C'\geq ||\Omega||^2={\sum_{n=0}^\infty\frac{1}{a_n'}\left(\prod_{i=n}^\infty|c_i|^2\right)}\geq {\sum_{n=0}^\infty\frac{1}{a_n'}\left(\prod_{i=0}^\infty|c_i|^2\right)}=\frac{C'}{K^2},
\end{equation*}
which implies that $|L(f)|\leq K\sqrt{C}||f||$ and $||\overline{T}||\leq \sqrt{CC'}+K\sqrt{CC'}$.
This completes the proof.

\end{proof}

We again get a corollary on the existence of limits at infinity for sequences which are in the domain of $\overline{A}$.

\begin{cor}\label{limitunibar}
Let $f\in\textrm{dom}(\overline{A})$, then $f_\infty$ exists and is given by the following formula

\begin{equation*}
f_\infty =  -L(Af).
\end{equation*}

\end{cor}

\begin{proof}
The proof for the $\overline{T_0}$ term is identical to the proof of the corollary (\ref{limituni}).  To compute the limit of the other term we note:
\begin{equation*}
\Omega_n=\prod_{i=n}^\infty \overline{c_i}=\frac{\prod_{i=0}^\infty \overline{c_i}}
{\prod_{i=0}^{n-1} \overline{c_i}}\to 1
\end{equation*}
as $n\to\infty$.
\end{proof}

The above corollaries allow us to consider ``boundary" conditions on $A$ and
$\overline{A}$. We define the operators $A_0$ and $\overline{A_0}$ as follows: $A_0$ is the operator $A$ but with domain
\begin{equation*}
\textrm{dom}(A_0) = \{f\in \textrm{dom}(A) : f_\infty = 0\},
\end{equation*}
and $\overline{A_0}$ is the operator $\overline{A}$ with domain
\begin{equation*}
\textrm{dom}(\overline{A_0}) = \{f\in \textrm{dom}(\overline{A}) : f_\infty = 0\}.
\end{equation*}

The four operators are closely related as shown by the following computation of the adjoint of $A$.

\begin{prop}\label{prelimsbi2}
The adjoint of $A$ has the following formula 

\begin{equation*}
A^* = \overline{A_0}.
\end{equation*}
Moreover the adjoint of $\overline{A}$ has the following formula

\begin{equation*}
\overline{A}^* = A_0.
\end{equation*}
\end{prop}

\begin{proof}
Computing the inner product one has:

\begin{equation*}
\di\langle Af, g \di\rangle = \di\sum_{n=0}^\infty \frac{1}{a_n}\overline{a_n(f_n - c_{n-1}f_{n-1})}g_n = \di\sum_{n=0}^\infty \overline{(f_n - c_{n-1}f_{n-1})}g_n =
\end{equation*}

\begin{equation*}
= \di\lim_{N\to\infty}\di\sum_{n=0}^N\overline{(f_n - c_{n-1}f_{n-1})}g_n = \di\lim_{N\to\infty}\left(\di\sum_{n=0}^N\overline{f_n}g_n - \di\sum_{n=0}^N\overline{c_{n-1}f_{n-1}}g_n\right).
\end{equation*}

Then, setting $n-1\mapsto n$ one arrives at

\begin{equation*}
\begin{aligned}
\langle Af, g \rangle &= \lim_{N\to\infty}\left(\sum_{n=0}^N\overline{f_n}(g_n - \overline{c_n}g_{n+1}) - c_N\overline{f_N}g_{N+1}\right)= \\
&= \sum_{n=0}^\infty\frac{1}{a_n'}\overline{f_n}a_n'(g_n - \overline{c_n}g_{n+1}) - \overline{f_\infty}g_\infty =\\
&= \langle f, \overline{A} g \rangle - \overline{f_\infty}g_\infty.
\end{aligned}
\end{equation*}
Here note that $\prod c_n^{-1} <\infty$ and $|c_n| \le 1$ implies that the $c_n$ converge to $1$.

The functional $f\to f_\infty$ is not continuous thus implying that if  $f\in \textrm{dom}(A^*)$, then $f_\infty = 0$ and if $g\in \textrm{dom}(\overline{A}^*)$, then $g_\infty = 0$.   This completes the proof.

\end{proof}

It follows that all four operators are Fredholm operators where parametrix in each case is
$T$, $\overline{T}$, or their adjoints.
For completeness we compute the adjoint of $T$ and of $\overline{T}$: this is not necessary for the main argument but may possibly be useful in future applications.

\begin{prop}\label{tstarcomp}
The adjoint of $T$ is equal to $\overline{T_0}$ of \ref{tbardef}, i.e. it has the following formula: 

\begin{equation*}
T^*f_n = \overline{T_0}f_n=\sum_{k=n}^\infty \frac{1}{a_k'}\left(\prod_{j=n}^{k-1}\overline{c_j}\right)f_k .
\end{equation*}
Similarly:

\begin{equation*}
\overline{T}^*f = Tf-\frac{\langle \Omega, f \rangle}{||\Omega||}\,T\Omega.
\end{equation*}

\end{prop}

\begin{proof}
Looking at the inner product one has

\begin{equation*}
\di\langle Tg, f \di\rangle = \di\sum_{n=0}^\infty \frac{1}{a_n}\overline{Tg_n}f_n = \di\sum_{n=0}^\infty \frac{1}{a_n}\di\left(\frac{1}{a_n}\overline{g_n} + \frac{\overline{c_{n-1}}}{a_{n-1}}\overline{g_{n-1}} + \cdots + \frac{\overline{c_{n-1}\cdots c_0}}{a_0}\overline{g_0}\di\right)f_n =
\end{equation*}

\begin{equation*}
= \di\sum_{n=0}^\infty\frac{1}{a_n}\di\left(\frac{1}{a_n}\overline{g_n}f_n\right) + \di\sum_{n=0}^\infty \frac{1}{a_n} \di\left(\frac{\overline{c_{n-1}}}{a_{n-1}}\overline{g_{n-1}}f_n\right) + \cdots + \di\sum_{n=0}^\infty \frac{1}{a_n} \di\left(\frac{\overline{c_{n-1}\cdots c_0}}{a_0}\overline{g_0}f_n \right).
\end{equation*}

   Then using $n\mapsto j+1$ in the second sum, $n\mapsto j+2$ in the third sum and so on and relabeling the indices, one has

\begin{equation*}
\di\langle Tg, f \di\rangle = \di\sum_{n=0}^\infty \frac{1}{a_n}\overline{g_n} \di\left( \frac{1}{a_n}f_n\right) + \di\sum_{n=0}^\infty \frac{1}{a_n}\overline{g_n}\di\left(\frac{\overline{c_n}}{a_{n+1}}f_{n+1}\right) + \cdots =
\end{equation*}

\begin{equation*}
= \di\sum_{n=0}^\infty \frac{1}{a_n'}\overline{g_n}\di\left(\frac{1}{a_n'}f_n \right) + \di\sum_{n=0}^\infty \frac{1}{a_n'}\overline{g_n}\di\left(\frac{\overline{c_n}}{a_{n+1}'}f_{n+1} \right) + \cdots + \di\sum_{n=0}^\infty \frac{1}{a_n'}\overline{g_n}\di\left(\frac{\overline{c_n \cdots c_{n+k}}}{a_{n+(k+1)}'}f_{n+(k+1)}\right) + \cdots =
\end{equation*}

\begin{equation*}
= \di\sum_{n=0}^\infty \frac{1}{a_n'}\overline{g_n}\di\left(\frac{1}{a_n'}f_n + \frac{\overline{c_n}}{a_{n+1}'}f_{n+1} + \cdots + \frac{\overline{c_n \cdots c_{n+k}}}{a_{n+(k+1)}'}f_{n+(k+1)} + \cdots \right) = \di\langle g, T^*f \di\rangle.
\end{equation*}
This then shows the first result. For the second formula we notice that we just showed that
$\overline{T_0}^*=T$ and the second term comes from an easy computation of the adjoint of the projection $f\to L(f)\Omega$.
\end{proof}

Combining propositions \ref{inverseun}, \ref{inverseunker}, \ref{prelimsbi2}
we get the following results about $A_0$ and $\overline{A_0}$.

\begin{cor}\label{inveresun}
$A_0$ is an unbounded Fredholm operator with index equal to minus one. We have

\begin{equation*}
\begin{aligned}
A_0T_0 &= I_{\wlltwonp} - \textrm{Proj}_{\textrm{Coker}(A_0)} \\
T_0A_0 &= I_{\wlltwon}
\end{aligned}
\end{equation*}
where $T_0:=\overline{T}^*$
\end{cor}

We also have:

\begin{cor}
$\overline{A_0}$ is an unbounded Fredholm operator with index zero, and

\begin{equation*}
\begin{aligned}
\overline{A_0}\overline{T_0} &= I_{\wlltwon} \\
\overline{T_0}\overline{A_0} &= I _{\wlltwonp}.
\end{aligned}
\end{equation*}

\end{cor}

It turns out that we can say more about the parametrices introduced above.

\begin{prop}
Each of the parametrix operators: $T$, $T_0$, $\overline{T}$, $\overline{T_0}$  is  a Hilbert-Schmidt operator.
\end{prop}

\begin{proof}
We present the details for the operator $T$, other cases are similar.
In fact the proposition already follows from the way we estimated the norm of $T$ since $T$ is an integral operator. We give an alternative proof here. First note that $\|T\|_{HS}^2 = tr(T^*T) = \di\sum_{i=0}^\infty\|Te_i\|^2$ where $\{e_i\}$ is the canonical basis for $\wlltwon.$ So 

\begin{equation*}
\di\left(Te_i\right)_n = \frac{1}{a_n}(e_i)_n + \frac{c_{n-1}}{a_{n-1}}(e_i)_{n-1} + \cdots + \frac{c_{n-1}\cdots c_0}{a_0}(e_i)_0 .
\end{equation*}
It follows that $(Te_i)_n = 0$ $\forall n < i$, and

\begin{equation*}
\begin{aligned}
(Te_i)_i &= \frac{\sqrt{a_i}}{a_i} \\
(Te_i)_{i+1} &= \frac{c_i}{a_i}\sqrt{a_i} \\ 
(Te_i)_{i+2} &= \frac{c_{i+1}c_i}{a_i}\sqrt{a_i} \\
&\vdots
\end{aligned}
\end{equation*}
Then we estimate

\begin{equation*}
\di\|Te_i\|^2 = \frac{1}{a_i}\di\sum_{k=0}^\infty\frac{1}{a_{i+k}'}|c_ic_{i+1}\cdots c_{i+k}|^2 \le \frac{1}{a_i}C' ,
\end{equation*}
and consequently

\begin{equation*}
\|T\|_{HS}^2 = \di\sum_{i=0}^\infty\|Te_i\|^2 \le C'\di\sum_{i=0}^\infty\frac{1}{a_i} \le CC' \ \ \Rightarrow \|T\|_{HS} \le \sqrt{CC'} .
\end{equation*}
\end{proof}

We now shift our attention to the bilateral case and study the same type of properties as considered in the unilateral case. It turns out that both $A$ and $\overline{A}$ have one dimensional  kernels in that case, one has to use infinite products for some expressions, and there are more options of imposing conditions at infinities. However the analytic aspects of the theory are no different then the unilateral case and so we provide less detail in some estimates to avoid repetitiveness.

\subsection{Bilateral Case}
As in the unilateral case we start with the study of the kernels of $A$ and $\overline{A}$.  It turns out that both $A$ and $\overline{A}$ have one dimensional kernels.  First recall the constants defined at the beginning of this section

\begin{equation*}
C = \sum_{n\in\Z} \frac{1}{a_n} <\infty \ , \ C'=\sum_{n\in\Z} \frac{1}{a_n'} <\infty \ \ \textrm{and} \ \ \ K = \prod_{n\in\Z} \frac{1}{|c_n|} < \infty .
\end{equation*}

\begin{prop}\label{prelimsbiker}
Given $A$ and $\overline{A}$ above we have:

\begin{equation*}
\begin{aligned}
\textrm{dim} \ \textrm{Ker}{A} &= 1 \\
\textrm{dim} \ \textrm{Ker}\overline{A} & = 1.
\end{aligned}
\end{equation*}

\end{prop}

\begin{proof}
First we  investigate the kernel of $A$.   To this end we need to solve the equation $Af_n = a_n(f_n - c_{n-1}f_{n-1}) = 0$ for $n \in \Z$.   This is done recursively and, for $n\geq 0$, one arrives at the following

\begin{equation*}
f_n = \di\left(\di\prod_{i = -1}^{n-1}c_i\right)f_{-1}, \ \ \ n \ge 0  .
\end{equation*}
Next, in a similar fashion, solve the equation for $n < 0$ to get the following

\begin{equation*}
f_{-n} = \di\left(\di\prod_{i= -2}^{-n}\frac{1}{c_i}\right)f_{-1}, \ \ \ n \ge 1  .
\end{equation*}
The two formulas above can be written compactly in the following semi-infinite product

\begin{equation*}
f_n = \left(\prod_{i=-\infty}^{n-1}c_i\right)\alpha
\end{equation*}
for any constant $\alpha$. To see that the kernel of $A$ is indeed one dimensional, we need to verify that $\{f_n\}\in\wlltwozp$.  Using the fact that $|c_i|\le 1$ for all $i$ one has that

\begin{equation*}
\|f\|_{\wlltwozp}^2 = \sum_{n\in\Z} \frac{1}{a_n'}\left|\prod_{i=-\infty}^{n-1}c_i\right|^2|\alpha|^2 \le 
|\alpha|^2 \sum_{n\in\Z} \frac{1}{a_n'} = |\alpha|^2C' < \infty,
\end{equation*}
thus $\{f_n\}\in\wlltwozp$.   

Next we study the equation $\overline{A}f_n = a_n'(f_n - \overline{c_n}f_{n+1}) = 0$ for $n\in\Z$. We get

\begin{equation*}
f_n = \left(\prod_{i=0}^{n-1}\frac{1}{\overline{c_i}}\right)f_0 \quad\textrm{for}\quad n\ge0
\end{equation*}
and the similar formula for $n<0$

\begin{equation*}
f_{-n} = \left(\prod_{i=1}^{-n} \overline{c_i}\right)f_0 \quad\textrm{for}\quad n\ge 1 .
\end{equation*}
We also have the same type of semi-infinite product for $\overline{A}$:

\begin{equation*}
f_n = \left(\prod_{i=n}^\infty\overline{c_i}\right)\beta
\end{equation*}
for any constant $\beta$.  As with $A$, to guarantee that the kernel of $\overline{A}$ is one dimensional, we need to verify that $\{f_n\}\in\wlltwoz$.  Using the fact that $|c_i|\le 1$ for all $i$ one has that

\begin{equation*}
\|f\|_{\wlltwoz}^2 = \sum_{n\in\Z} \frac{1}{a_n}\left|\prod_{i=n}^\infty\overline{c_i}\right|^2|\beta|^2 \le  |\beta|^2 \sum_{n\in\Z} \frac{1}{a_n} = |\beta|^2C < \infty.
\end{equation*}
This completes the proof.
\end{proof}

Next we construct a parametrix for $A$.
\begin{prop}\label{prelimsbi1}
There exists a $T \in B(\wlltwoz, \wlltwozp)$ such that $AT = I_{\wlltwoz}$ and $TA = I_{\wlltwozp} - Proj_{Ker A}$. In particular $A$ is an unbounded Fredholm operator with index equal to one.
\end{prop}

\begin{proof}
We start by looking at the equation $Af_n = a_n(f_n - c_{n-1}f_{n-1}) = g_n$ which can be written as:

\begin{equation}\label{star1}
f_n - c_{n-1}f_{n-1} = \frac{g_n}{a_n} .
\end{equation}
We use variation of constants method to solve ($\ref{star1}$).  First observe that the homogeneous equation
$
f_n-c_{n-1}f_{n-1} = 0
$
has the following solution by the kernel calculation in proposition ($\ref{prelimsbiker}$):
$f_n = \left(\prod_{i=-\infty}^{n-1}c_i\right)\alpha$
for some constant $\alpha$. Consequently we set 
\begin{equation*}
f_n = \left(\prod_{j=-\infty}^{n-1}c_j\right)\alpha_n
\end{equation*}
and substitute this into equation ($\ref{star1}$). This leads to the following equation for 
$\alpha_n$:

\begin{equation*}
\alpha_n - \alpha_{n-1} = \left(\prod_{j=-\infty}^{n-1}\frac{1}{c_j}\right)\frac{g_n}{a_n}
\end{equation*}
which has a  solution given by:

\begin{equation*}
\alpha_n = \sum_{i=-\infty}^n \frac{1}{a_i}\left(\prod_{j=-\infty}^{i-1}\frac{1}{c_j}\right)g_i.
\end{equation*}
Therefore one has a particular solution of equation \ref{star1}:

\begin{equation*}
f_n = \sum_{i=-\infty}^n \frac{1}{a_i}\left(\prod_{j=i}^{n-1}c_j\right)g_i,
\end{equation*}
and the general solution is

\begin{equation*}
f_n = \sum_{i=-\infty}^n \frac{1}{a_i}\left(\prod_{j=i}^{n-1}c_j\right)g_i - \left(\prod_{i=-\infty}^{n-1}c_i\right)\alpha.
\end{equation*}
The above expression gives the formula for $T$:

\begin{equation}\label{tdefbil}
Tg_n = T_1g_n-\alpha(g)\Omega^-_n,
\end{equation}
where
\begin{equation}\label{t1defbil}
 T_1g_n:=\sum_{i=-\infty}^n \frac{1}{a_i}\left(\prod_{j=i}^{n-1}c_j\right)g_i
\end{equation}
and $\Omega^-_n:=\prod_{i=-\infty}^{n-1}c_i$, and $\alpha(g)$ arbitrary.

It's is clear from our construction that $AT = I_{\wlltwoz}$. To make sure that we get $TA = I_{\wlltwozp} - \textrm{Proj}_{\textrm{Ker }A}$, we must make a choice on $\alpha(g)$ just as in the unilateral case:

\begin{equation*}
\alpha(g) := \frac{\langle \Omega^-, {T_1}g\rangle}{||\Omega^-||^2}
= \frac{\sum_{n\in\Z} \sum_{i=-\infty}^n \frac{1}{a_n}\frac{1}{a_i}\left(\prod_{k=-\infty}^{n-1}\overline{c_k}\right)\left(\prod_{j=i}^{n-1}c_j\right)g_i}{\sum_{n\in\Z}\frac{1}{a_n}\left(\prod_{i=-\infty}^{n-1}|c_i|^2\right)} .
\end{equation*}

Convergence of the sums and products and the boundedness of $T$ is established just as in the unilateral case.
The operator $\overline{T_1}$ is bounded by $\sqrt{CC'}$ in essentially the same way as the operator $T$ is proposition (\ref{inverseun}). To see that we write
\begin{equation*}
(T_1g)_n = \frac{\sqrt{a_n}}{a_n}\frac{1}{\sqrt{a_n}}g_n + \frac{c_{n-1}\sqrt{a_{n-1}}}{a_{n-1}}\frac{1}{\sqrt{a_{n-1}}}g_{n-1} + \cdots - 
\end{equation*}
and estimate using the Cauchy-Schwartz inequality and the fact that the $|c_i| \le 1$ for all $i$:

\begin{equation*}
\begin{aligned}
|T_1g_n|^2 &\le \left[\left(\frac{\sqrt{a_n}}{a_n}\right)^2 + \left(\frac{\sqrt{a_{n-1}}}{a_{n-1}}\right)^2 + \cdots\right]\left(\frac{1}{a_n}|g_n|^2 + \frac{1}{a_{n-1}}|g_{n-1}|^2 + \cdots\right) \\
&\le \left(\sum_{i=-\infty}^n \frac{1}{a_i}\right)\|g\|^2 .
\end{aligned}
\end{equation*}
Consequently

\begin{equation*}
\|T_1g\|^2 =\sum_{n\in\Z} \frac{1}{a_n'}|T_1g_n|^2
\le \sum_{n\in\Z} \frac{1}{a_n'} C\|g\|^2
= (C'C)\|g\|^2 .
\end{equation*}

To estimate $\alpha(g)$ we notice that
\begin{equation*}
C'\geq ||\Omega||^2={\sum_{n\in\Z}\frac{1}{a_n'}\left(\prod_{i=n}^\infty|c_i|^2\right)}\geq {\sum_{n\in\Z}\frac{1}{a_n'}\left(\prod_{i\in\Z}|c_i|^2\right)}=\frac{C'}{K^2},
\end{equation*}
which implies that $|\alpha(g)|\leq K\sqrt{C}||g||$ and $||T||\leq \sqrt{CC'}+K\sqrt{CC'}$.
This completes the proof.
\end{proof}

An important corollary from this proposition is the existence of limits at infinities for the sequences which are in the domain of $A$.

\begin{cor}\label{limitbi}
Let $f\in\textrm{dom}(A)$, then $f_{\pm\infty}$ exist and are given by the following formulas:

\begin{equation*}
\begin{aligned}
f_\infty &= \sum_{i=-\infty}^\infty \frac{1}{a_i}\left(\prod_{j=i}^\infty c_j \right)Af_i - \left(\prod_{i=-\infty}^\infty c_i\right)\alpha(Af) \\
f_{-\infty} &= \alpha(Af).
\end{aligned}
\end{equation*}

\end{cor}

\begin{proof}
Using the previous proposition and the methods invoked in corollaries (\ref{limituni}) and (\ref{limitunibar}) yields the desired result.
\end{proof}

Next we state analogous results about the $\overline{A}$.

\begin{prop}\label{inversebiabar}
There exists a $\overline{T} \in B(\wlltwozp, \wlltwoz)$ such that $\overline{A}\overline{T} = I_{\wlltwozp}$ and $\overline{T}\overline{A} = I_{\wlltwozp} - Proj_{Ker \overline{A}}$.
In particular $\overline{A}$ is an unbounded Fredholm operator with index equal to one.
\end{prop}

\begin{proof}
The solution of the equation

\begin{equation*}
a_n'(f_n - \overline{c_n}f_{n+1}) = g_n \quad\textrm{for} \ n\in\Z
\end{equation*}
is given the following formula

\begin{equation}\label{tbardefbil}
\overline{T}g_n = \sum_{i=n}^\infty \frac{1}{a_i'}\left(\prod_{j=n}^{i-1}\overline{c_j}\right)g_i -\left(\prod_{i=n}^\infty \overline{c_i}\right)\beta(g)=\overline{T_0}g_n-\beta(g)\Omega^+_n,
\end{equation}
where we set $\prod_{j=n}^{n-1}\overline{c_j} = 1$  and $\beta(g)$ is an arbitrary constant.  Here
\begin{equation}\label{t0bardefbil}
\overline{T_0}g_n:=\sum_{i=n}^\infty \frac{1}{a_i'}\left(\prod_{j=n}^{i-1}\overline{c_j}\right)g_i,
\end{equation}
and $\Omega^+_n:=\prod_{i=n}^\infty \overline{c_i}$.

One has the relation $\overline{A}\overline{T} = I_{\wlltwozp}$, however to make sure one has $\overline{T}\overline{A} = I_{\wlltwoz} - \textrm{Proj}_{\textrm{Ker }\overline{A}}$, we need to make the following choice of $\beta(g)$:   

\begin{equation*}
\beta(g) = \frac{\langle \Omega^+, \overline{T_0}g\rangle}{||\Omega^+||^2}
= \frac{\sum_{n\in\Z}\sum_{i=n}^\infty\frac{1}{a_n'}\frac{1}{a_i'}\left(\prod_{j=n}^{i-1}c_j\right)\left(\prod_{k=n}^\infty \overline{c_k}\right)g_i}{\sum_{n\in\Z}\frac{1}{a_n'}\left(\prod_{i=n}^\infty|c_i|^2\right)} .
\end{equation*}
The previous methods yield
\begin{equation*}
\|\overline{T}\| \le \sqrt{CC'}+K\sqrt{CC'} < \infty,
\end{equation*}
and the statement of the proposition follows.
\end{proof}

An immediate corollary is the following:

\begin{cor}
Let $f\in\textrm{dom}(\overline{T})$, then $f_{\pm\infty}$ exist and 
\begin{equation*}
\begin{aligned}
f_\infty &= \beta(Af) \\
f_{-\infty} &= \sum_{i=-\infty}^\infty \frac{1}{a_i'}\left(\prod_{j=-\infty}^{i-1}\overline{c_j}\right)Af_i -\left(\prod_{i=-\infty}^\infty \overline{c_i}\right)\beta(Af).
\end{aligned}
\end{equation*}

\end{cor}

Imposing vanishing conditions at infinities we can construct the following six operators.
$A_0$ is the operator $A$ but with domain

\begin{equation*}
\textrm{dom}(A_0) = \{f\in \textrm{dom}(A) : f_\infty = 0\}
\end{equation*}
and $\overline{A_0}$ is the operator $\overline{A}$ with domain

\begin{equation*}
\textrm{dom}(\overline{A_0}) = \{f\in \textrm{dom}(\overline{A}) : f_\infty = 0\} .
\end{equation*}

$A_1$ is the operator $A$ with domain

\begin{equation*}
\textrm{dom}(A_1) = \{f\in \textrm{dom}(A) : f_{-\infty} = 0\}
\end{equation*}
and $\overline{A_1}$ is the operator $\overline{A}$ with domain

\begin{equation*}
\textrm{dom}(\overline{A_1}) = \{f\in \textrm{dom}(\overline{A}) : f_{-\infty} = 0\} .
\end{equation*}

Finally $A_2$ is the operator $A$ with domain

\begin{equation*}
\textrm{dom}(A_2) = \{f\in \textrm{dom}(A) : f_{\pm\infty} = 0\}
\end{equation*}
and $\overline{A_2}$ is the operator $\overline{A}$ with domain

\begin{equation*}
\textrm{dom}(\overline{A_2}) = \{f\in \textrm{dom}(\overline{A}) : f_{\pm\infty} = 0\} .
\end{equation*}

The above operators are related by the calculation of adjoints of $A$ and $\overline{A}$.

\begin{prop}\label{Dstarbi}
With the above definitions we have:

\begin{equation*}
A^* = \overline{A_2},\ 
A_0^* = \overline{A_1},\ 
A_1^* = \overline{A_0},\ 
A_2^* = \overline{A},\ 
\overline{A}^* = A_2,\ 
\overline{A_0}^* = A_1,\ 
\overline{A_1}^* = A_0,\ 
\overline{A_2}^* = A.
\end{equation*}
\end{prop}
\begin{proof}
This easily follows from the integration by parts formula:
\begin{equation*}
\langle Af, g \rangle = \langle f, \overline{A} g \rangle - \overline{f_\infty}g_\infty+
\overline{f_{-\infty}}g_{-\infty}.
\end{equation*}
\end{proof}

It follows from the definitions and the kernel calculations for $A$ and $\overline{A}$ that the just introduced six operators $A_0, A_1, A_2, \overline{A_0}, \overline{A_1}, \overline{A_2}$ have no kernel, while the adjoint calculation shows that only $A_2, \overline{A_2}$ have cokernel (of dimension one).

Next we find  a parametrix for each of the above operators. So far we have constructed $T$, formula \ref{tdefbil}, and
$\overline{T}$, formula \ref{tbardefbil}. In view of the above proposition we set $T_2:=\overline{T}^*$ and
$\overline{T_2}:= T^*$. We have also introduced $T_1$, formula \ref{t1defbil}, and
$\overline{T_0}$, formula \ref{t0bardefbil} and one can verify like in proposition \ref{tstarcomp} that
$T_1^*=\overline{T_0}$. We introduce similar looking operators:
\begin{equation*}
\overline{T_1}g_n:=\sum_{i=-\infty}^n \frac{1}{a_i}\left(\prod_{j=i}^{n-1}\overline{c_j}\right)g_i
\end{equation*}
and
\begin{equation*}
{T_0}g_n:=\sum_{i=n}^\infty \frac{1}{a_i'}\left(\prod_{j=n}^{i-1}{c_j}\right)g_i,
\end{equation*}
for which we have $T_0^*=\overline{T_1}$.
Then we get the following summary of the Fredholm properties of our operators.

\begin{prop}\label{inversesbi}
With the above definitions we have

\begin{equation*}
\begin{aligned}
A_0T_0 &= I_{\wlltwoz}  \textrm{   and   } T_0A_0 = I_{\wlltwozp}\\
A_1T_1 &= I_{\wlltwoz}  \textrm{   and   } T_1A_1 = I_{\wlltwozp}\\
A_2T_2 &= I_{\wlltwoz}   - Proj_{Coker(A_2)}\textrm{   and   } T_2A_2 = I_{\wlltwozp}\\
\overline{T_0A_0} &= I_{\wlltwozp} \textrm{   and   } \overline{A_0T_0} = I_{\wlltwoz} \\
\overline{T_1A_1} &= I_{\wlltwozp} \textrm{   and   } \overline{A_1T_1} = I_{\wlltwoz} \\
\overline{T_2A_2} &= I_{\wlltwozp} \textrm{   and  } \overline{A_2T_2} = I_{\wlltwoz} - Proj_{Coker(\overline{A_2)}}.\\
\end{aligned}
\end{equation*}
In particular all six operators are unbounded Fredholm operators with index zero for $A_0$, $A_1$, $\overline{A_0}$, $\overline{A_1}$
and index minus one for $A_2$, $\overline{A_2}$.
\end{prop}

We conclude this section with a simple observation on functional-analytic properties of the parametrices.

\begin{prop}
Each of the 8 parametrix operators: $T$, $T_0$, $T_1$, $T_2$, $\overline{T}$, $\overline{T_0}$, $\overline{T_1}$, $\overline{T_2}$  is  a Hilbert-Schmidt operator.
\end{prop}

\section{Fourier Transform in Quantum Domains}
In this section we consider the Fourier Transform in the quantum domains, and get decomposition theorems for the Hilbert Space $\mathcal{H}$ and the operator $D$, defined in Section 3.
The following discussion covers both cases ${\mathbb S}=\N$  and ${\mathbb S}=\Z$ in a fairly uniform manner: there are only a few places where the difference between the unilateral and the bilateral cases needs to be covered separately. 
We will make an extensive use of the label operator defined as:

\begin{equation*}
Ke_k = ke_k,
\end{equation*}
where $\{e_k\}$, $k\in {\mathbb S}$ is the canonical basis for $\wlltwos$. The label operator lets us write different diagonal operators as its functions. For example two previously introduced operators can be expressed, with some notational abuse, as $W=W(K)$,  and $S=S(K)$, see \ref{Wdefref} and \ref{Sdefref},
with $W(k)=w_k$, and $S(k)=s_k=w_k^2-w_{k-1}^2$. Additionally, the elements of $\wlltwos$ will  also be written using the function notation i.e. $\{f_k\}=\{f(k)\}$. If $\{f(k)\}$ has a limit at $\pm\infty$ it is denoted by
$f(\pm\infty)$.

For the purpose of the following discussion we define
\begin{equation}\label{andef}
a^{(n)}(k) = S^{-1/2}(k)S^{-1/2}(k+n). \\
\end{equation}
Then one has the following lemma which is essentially a Fourier decomposition of the Hilbert space $\mathcal{H}$.

\begin{lem}\label{hilbert}
Let $a^{(n)} = \{a^{(n)}(k)\}$  be the  sequence of positive numbers defined above. The map $I:\bigoplus_{m=0}^\infty \ell_{{a}^{(m)}}^2(\mathbb{S}) \oplus\bigoplus_{n=1}^\infty \wlltwosc \rightarrow \mathcal{H}$ given by

\begin{equation*}
\bigoplus_{m=0}^\infty\{f_m(k)\}_{k\in{\mathbb S}}\oplus\bigoplus_{n=1}^\infty\{g_n(k)\}_{k\in{\mathbb S}}\overset{I}{\rightarrow} \sum_{m=0}^\infty U^mf_m(K) + \sum_{n=1}^\infty g_n(K)(U^*)^n
\end{equation*}
is well-defined and is an isomorphism of Hilbert spaces.

\end{lem}

\begin{proof}
First we need to show that $I$ is an isometry.  We will only do this for the $g_n(K)$ terms as the calculation for the $f_n(K)$ terms is essentially identical.   We have
\begin{equation*}
\begin{aligned}
\left\|\sum_{n=1}^\infty g_n(K)(U^*)^n\right\|_{\mathcal{H}}^2 &= \textrm{tr}\left(S^{1/2}(K)\sum_{n=1}^\infty g_n(K)(U^*)^nS^{1/2}(K)\sum_{l=1}^\infty U^n\overline{g_l(K)}\right) \\
&=\textrm{tr}\left(S^{1/2}(K)S^{1/2}(K+n)\sum_{n=1}^\infty |g_n(K)|^2 \right)\\
&=\sum_{n=1}^\infty\sum_{k=0}^\infty\frac{1}{a_k^{(n)}}\left|g_n(k)\right|^2 
=\sum_{n=1}^\infty\left\|\{g_n(k)\}\right\|_{\ell_{a^{(n)}}^2}^2 = \|\{g_n(k)\}\|_{\bigoplus_{n=1}^\infty\ell_{a^{(n)}}^2}^2
\end{aligned}
\end{equation*}
and thus the norms are the same and $I$ is an isometry on its range. To show that $\textrm{Ran} \ I = \mathcal{H}$ we need to demonstrate that $\textrm{Ran} \ I$ is dense in $\mathcal{H}$.

First note that $C^*(W)$ is dense in $\mathcal{H}$ by construction.   Define $\delta_l(k)$ to be the following
function:
\begin{equation*}
\delta_l(k) = \left\{
\begin{array}{c}
1 \quad k=l \\
0 \quad k \neq l.
\end{array}\right. 
\end{equation*}
Then the (not normalized) canonical basis in $\ell_{{a}^{(m)}}^2(\mathbb{S})$ corresponds through the map $I$ to $U^m\delta_l(K)$ and  similarly the canonical basis in $\wlltwosc$ corresponds to $\delta_l(K)(U^*)^n$.   Note that $U^m\delta_l(K)$ and $\delta_l(K)(U^*)^n$ sit inside $C^*(W)$, so all that is required is to show they generate a dense set in $C^*(W)$ in the topology induced by $\mathcal{H}$ (they do not in the usual topology of $C^*(W)$). However this is clear since 

\begin{equation*}
\sum_{l\le L}\delta_l(K) \underset{L\to\infty}{\to} I \quad\textrm{in }\mathcal{H} 
\end{equation*}
because the operator $S$ is trace class.
It follows that $U,U^*$ are in $\textrm{Ran}\ I$, and thus  $\textrm{Ran}\ I$ is a dense subspace of $\mathcal{H}$.
\end{proof}

In what follows it will be convenient sometimes to write the Fourier series for $a\in\mathcal{H}$ in one of two ways:
\begin{equation*}
a=\sum_{m=0}^\infty U^mf_m(K) + \sum_{n=1}^\infty g_n(K)(U^*)^n=\sum_{m=1}^\infty U^mf_m(K) + \sum_{n=0}^\infty g_n(K)(U^*)^n
\end{equation*}
where we always set $f_0(k)=g_0(k)$.

We will now use the Fourier transform described in the above lemma to find a decomposition of $D$ in terms of the operators $A$ and $\overline{A}$ defined in the previous section. Recall that those operators depend on sequences of weights $a, a'$ and coefficients $c$ subject to conditions \ref{acconditions}. Since in the following the parameters vary, we will need appropriate decorations on $A$ and $\overline{A}$. 
To do that, in addition to sequences \ref{andef}, we introduce:
\begin{equation}\label{cndef}
c^{(n)}(k) := W(k)W^{-1}(k+n+1).
\end{equation}
Now we define the operators $A^{(n)}$  as follows:

\begin{equation*}
\begin{aligned}
&A^{(n)} : \textrm{dom}(A^{(n)}) \subset \ell_{a^{(n+1)}}^2(\mathbb{S}) \to \ell_{a^{(n)}}^2(\mathbb{S}) \\
&\textrm{where} \ \textrm{dom}(A^{(n)}) = \{f\in\ell_{a^{(n+1)}}^2(\mathbb{S})\ : \ \|Af\|_{\ell_{a^{(n)}}^2(\mathbb{S})}<\infty\} \\
&A^{(n)}f(k) = a^{(n)}(k)\left(f(k)-c^{(n)}(k-1)\,f(k-1)\right). \\
\end{aligned}
\end{equation*}
The corresponding formal adjoints $\overline{A}^{(n)}$ are defined in the same way as in the previous section i.e.

\begin{equation*}
\begin{aligned}
&\overline{A}^{(n)} : \textrm{dom}(\overline{A}^{(n)})\subset \ell_{{a}^{(n)}}^2(\mathbb{S}) \to \ell_{{a}^{(n+1)}}^2(\mathbb{S}) \\
&\textrm{where} \ \textrm{dom}(\overline{A}^{(n)}) = \{f\in\wlltwosc \ : \ \|\overline{A}f\|_{\ell_{{a}^{(n+1)}}^2(\mathbb{S})} <\infty\} \\
&\overline{A}^{(n)}f(k) = {a}^{(n+1)}(k)(f(k) - \overline{c}^{(n)}(k)f(k+1)).
\end{aligned}
\end{equation*}
Additionally we will need the following diagonal operator $W^{(m)}(K):=W(K+m)$ i.e. 
\begin{equation*}
W^{(m)}f(k):=W(k+m)f(k)
\end{equation*}
for $f\in \ell_{{a}^{(n)}}^2(\mathbb{S})$. Clearly $W^{(m)}$ is a bounded, invertible, self-adjoint operator with a bounded inverse.

Now we can state the main decomposition theorem. A minor difficulty here is that $D$ is not diagonal with respect to the Fourier decomposition of the Hilbert space but rather shifts the components by one.

\begin{theo}\label{opdecomp}
With the above notation the operator $D$ has the following decomposition: $Da= \di\sum_{m=1}^\infty U^mf'_m(K) + \sum_{n=0}^\infty g'_n(K)(U^*)^n$, where $a = \di\sum_{m=0}^\infty U^mf_m(K) + \sum_{n=1}^\infty g_n(K)(U^*)^n$ and $f_{m+1}'= -\overline{ A}^{(m)}W^{(m)}f_m$
and $g_{n-1}'= W^{(n-1)}A^{(n-1)}g_n$. We write symbolically:
\begin{equation*}
D \cong  \left( (-\overline{ A}^{(m)}W^{(m)})_{m=0}^\infty,
(W^{(n-1)}A^{(n-1)})_{n=1}^\infty \right).
\end{equation*}

\end{theo}

\begin{proof}
We compute the expression $Da = S^{-1/2}(K)\left[a,UW(K)\right]S^{-1/2}(K)$ using the Fourier decomposition: $a = \di\sum_{m=0}^\infty U^mf_m(K) + \sum_{n=1}^\infty g_n(K)(U^*)^n$.
We use the following commutation relation

\begin{equation*}
f(K)U = Uf(K+1) .
\end{equation*}
Then one obtains, setting in the unilateral case $W(-1)=f_n(-1)=g_n(-1)=0$,

\begin{equation*}
\begin{aligned}
&Da = S^{-1/2}(K)\left[a,UW(K)\right]S^{-1/2}(K) \\
&= \sum_{m=0}^\infty S^{-1/2}(K)\left(U^mf_m(K)UW(K) - UW(K)U^mf_m(K)\right)S^{-1/2}(K) \\
&+ \sum_{n=1}^\infty S^{-1/2}(K)\left(g_n(K)(U^*)^{n-1}W(K) - UW(K)g_n(K)(U^*)^n\right)S^{-1/2}(K). \\
\end{aligned}
\end{equation*}
The above expression is equal to
\begin{equation*}
\begin{aligned}
&-\sum_{m=0}^\infty U^{m+1}S^{-1/2}(K)S^{-1/2}(K+m+1)\left(W(K+m)f_m(K) - W(K)f_m(K+1)\right) \\
&+\sum_{n=1}^\infty S^{-1/2}(K)S^{-1/2}(K+n-1)\left(W(K+n-1)g_n(K)-W(K-1)g_n(K-1)\right)(U^*)^{n-1},\\
\end{aligned}
\end{equation*}
which can be written as:
\begin{equation*}
\begin{aligned}
&-\sum_{m=0}^\infty U^{m+1}a^{(m+1)}(K)\left(W(K+m)f_m(K) - \frac{W(K)}{W(K+m+1)}W(K+m+1)f_m(K+1)\right) \\
&+\sum_{n=1}^\infty W(K+n-1)a^{(n-1)}(K)\left(g_n(K)-\frac{W(K-1)}{W^(K+n-1)}g_n(K-1)\right)(U^*)^{n-1}.\\
\end{aligned}
\end{equation*}
This is equal to:
\begin{equation*}
\begin{aligned}
&-\sum_{m=0}^\infty U^{m+1} a^{(m+1)}(K)\left(W^{(m)}(K)f_m(K) - c^{(m)}(K)W^{(m)}(K+1)f_m(K+1)\right) \\
&+\sum_{n=1}^\infty W^{(n-1)}(K)a^{(n-1)}(K)\left(g_n(K)-c^{(n-1)}(K-1)g_n(K-1)\right)(U^*)^{n-1}.\\
\end{aligned}
\end{equation*}
Consequently
\begin{equation*}
\begin{aligned}
&Da=-\sum_{m=0}^\infty  U^{m+1}\overline{ A}^{(m)}W^{(m)}f_m(K)
+\sum_{n=1}^\infty   W^{(n-1)}A^{(n-1)}g_n(K) (U^*)^{n-1}.\\
\end{aligned}
\end{equation*}
Next we need to verify that the $a^{(n)}$, see (\ref{andef}), and the $c^{(n)}$, see (\ref{cndef}), satisfy the conditions \ref{acconditions}.
Note that since $w_k$ is an increasing sequence converging to $w^+ > 0$ one has  $|c^{(n)}(k)| = \left|\frac{w_k}{w_{k+n+1}}\right|\le 1$.

In the unilateral case, $\mathbb{S}=\N$, we compute
\begin{equation*}
K^{(n)}:=\prod_{k=0}^\infty \frac{1}{c^{(n)}(k)} = \frac{(w^+)^{n+1}}{w_0\cdots w_{n}}<\infty.
\end{equation*}
Next note that

\begin{equation*}
\begin{aligned}
C^{(n)}&:=\sum_{k=0}^\infty \frac{1}{a^{(n)}(k)}= \sum_{k=0}^\infty \sqrt{s_ks_{k+n}}
\leq\sqrt{\sum_{k=0}^\infty s_k}\sqrt{\sum_{k=0}^\infty s_{k+n}}\\
&=\sqrt{w^+}\sqrt{\sum_{k=n}^\infty s_{k}}<\infty,\\
\end{aligned}
\end{equation*}
with the constant $C^{(n)}$ going to zero as $n\to\infty$. 

In the  bilateral case ($k\in\Z$) we have
\begin{equation*}
K^{(n)}:=\prod_{k=-\infty}^\infty \frac{1}{c^{(n)}(k)} = \frac{(w^+)^{n+1}}{(w^-)^{n+1}}<\infty.
\end{equation*}
Next we estimate

\begin{equation*}
\begin{aligned}
C^{(n)}&:=\sum_{k=-\infty}^\infty \frac{1}{a^{(n)}(k)}= \sum_{k\leq -n/2}\sqrt{s_ks_{k+n}}+\sum_{k> -n/2}\sqrt{s_ks_{k+n}}\\
&\leq\sqrt{w^+-w^-}\,\sqrt{\sum_{k\leq -n/2}s_k}+
\sqrt{w^+-w^-}\,\sqrt{\sum_{k>n/2} s_{k}}<\infty,\\
\end{aligned}
\end{equation*}
and again the constant $C^{(n)}$ goes to zero as $n\to\infty$. 
\end{proof}


As we will see later on, the significance of $\lim\limits_{n\to\infty} C^{(n)} = 0$
is that it implies compactness of a parametrix of $D$, subject to APS boundary conditions.

We state here without a proof the analogous result for the formal adjoint $\overline{D}$ of $D$.
We define
\begin{equation*}
\overline{D}b := S^{-1/2}(K)[b,W(K)U^*]S^{-1/2}(K).
\end{equation*}
on the maximal domain, like the operator $D$.
We have the following decomposition.

\begin{theo}\label{opbardecomp}
With the above notation the operator $\overline{D}$ can be written as
\begin{equation*}
\overline{D} \cong  \left( (-W^{(m)} A^{(m)})_{m=0}^\infty,
(\overline{A}^{(n-1)}W^{(n-1)})_{n=1}^\infty \right).
\end{equation*}
\end{theo}

\section{Results}

We are now in a position to consider the proofs of the main results of this paper.
We rephrase here the statements of the theorems from Section 3  adding more detail.
The operator $D_N$ equals the unilateral operator $D$ with domain

\begin{equation*}
\textrm{dom}(D_N) = \left\{ a\in \textrm{Dom}(D) : \ r(a)\in \textrm{Ran} \ P_N \right\} .
\end{equation*}

We will now prove the first of the main results of this paper.

\begin{theo}\label{APSdisk}
The operator $D_N$ defined above is an unbounded Fredholm operator with index $ind(D_N) = N+1$. In fact, there is a bounded operator $Q_N$ such that $Ker(Q_N)=Coker(D_N)$,
$D_NQ_N= I   - Proj_{Coker(D_N)}$,  and   $Q_ND_N = I- Proj_{Ker(D_N)}$.
Moreover the parametrix $Q_N$ is a compact operator.
\end{theo}

\begin{proof}
All the hard work has been done.  It's now just a matter of piecing together appropriate results from the previous sections.  First we analyze the APS boundary conditions.   Let 
$a = \sum_{n=0}^\infty U^n f_n(K) + \sum_{n=1}^\infty g_n(K)(U^*)^n$ be in $\textrm{dom}(D_N)$.
Then the restriction  $r(a)$ from section 3 is well defined. We note that $r$ acts on $U$, $U^*$, and $f(K)$ in the following way

\begin{equation*}
\begin{aligned}
r(U) &= e^{i\f} \\
r(U^*) &= e^{-i\f} \\
r(f(K)) &= f(\infty)\cdot I := \lim_{k\to\infty} f(k)\cdot I .
\end{aligned}
\end{equation*}
The third equation holds because the difference $f(K)-f(\infty)\cdot I$ is a compact operator,  and $r$ vanishes on compact operators.   Consequently we see that $r$ acts on $a\in \textrm{Dom}(D)$ in the following way:

\begin{equation*}
r(a) = \sum_{m=0}^\infty e^{im\f}f_n(\infty) + \sum_{n=1}^\infty g_n(\infty) e^{-in\f}.
\end{equation*}
This means that for $r(a)$ to be in the range of $P_N$, where $\textrm{Ran }P_N = \underset{n\le N}{\textrm{span}}\{e^{in\f}\}$, one has the following: if $N\ge 0$, then $f_n(\infty) =0$ for $n>N$, and if $N < 0$, then $f_n(\infty) = 0$ for all $n$ and $g_n(\infty) = 0$ for $n<-N$.   Thus from Theorem (\ref{opdecomp}) and from proposition (\ref{prelimsbi2}) one can represent $D_N$ subject to the APS boundary conditions as follows

\begin{equation*}
D_N = \left\{
\begin{array}{cc}
\left( (-\overline{ A}^{(m)}W^{(m)})_{m=0}^N,(-\overline{ A_0}^{(m)}W^{(m)})_{m=N+1}^\infty,
(W^{(n-1)}A^{(n-1)})_{n=1}^\infty \right) & \textrm{for} \ N\ge 0 \\
\left( (-\overline{ A_0}^{(m)}W^{(m)})_{m=0}^\infty,
(W^{(n-1)}A_0^{(n-1)})_{n=1}^{-N-1},(W^{(n-1)}A^{(n-1)})_{n=-N}^\infty \right) & \textrm{for} \ N<0
\end{array}\right.
\end{equation*}

Also note from Theorem (\ref{opdecomp}),  proposition (\ref{prelimsbi2}) and the above analysis of the APS conditions, one can represent $D_N^*$  as follows

\begin{equation*}
{D_N}^* = \left\{
\begin{array}{cc}
\left( (-W^{(m)} A_0^{(m)})_{m=0}^N, (-W^{(m)} A^{(m)})_{m=N+1}^\infty,
(\overline{A_0}^{(n-1)}W^{(n-1)})_{n=1}^\infty \right) & \textrm{for} \ N \ge 0 \\
\left( (-W^{(m)} A^{(m)})_{m=0}^\infty,
(\overline{A}^{(n-1)}W^{(n-1)})_{n=1}^{-N-1}, 
(\overline{A_0}^{(n-1)}W^{(n-1)})_{n=-N}^\infty \right) & \textrm{for} \ N <0
\end{array}\right.
\end{equation*}

From these representations and from proposition (\ref{prelimsker}), one gets the following

\begin{equation*}
\textrm{dim Ker} D_N = \left\{
\begin{array}{cc}
N+1 & \textrm{ for } N\ge0 \\
0 & \textrm{ for } N<0
\end{array}\right.
\end{equation*}
and

\begin{equation*}
\textrm{dim Ker} {D_N}^* = \left\{
\begin{array}{cc}
0 & \textrm{ for } N\ge0 \\
-(N+1) & \textrm{ for } N<0
\end{array}\right.
\end{equation*}
and thus the index calculation follows.  To conclude that $D_N$ is a Fredholm operator we need to construct a parametrix. We build $Q_N$ in the following fashion:

\begin{equation*}
Q_N = \left\{
\begin{array}{cc}
\left( (-V^{(m)} \overline{T}^{(m)})_{m=0}^N, (-V^{(m)} \overline{T_0}^{(m)})_{m=N+1}^\infty,
({T}^{(n-1)}V^{(n-1)})_{n=1}^\infty \right) & \textrm{for} \ N \ge 0 \\
\left( (-V^{(m)} \overline{T_0}^{(m)})_{m=0}^\infty,
({T_0}^{(n-1)}V^{(n-1)})_{n=1}^{-N-1}, 
({T}^{(n-1)}V^{(n-1)})_{n=-N}^\infty \right) & \textrm{for} \ N <0
\end{array}\right.
\end{equation*}
where  $T^{(n)}$, $\overline{T}^{(n)}$, $T_0^{(n)}$, and $\overline{T_0}^{(n)}$ are, correspondingly, the parametrices for $A^{(n)}$, $\overline{A}^{(n)}$, $A_0^{(n)}$ and $\overline{A_0}^{(n)}$, as defined in Section 3, and 
\begin{equation*}
V^{(m)}:=\left(W^{(m)}\right)^{-1} .
\end{equation*}
From corollary (\ref{inveresun}) and propositions (\ref{inverseunker}) and (\ref{inverseun}), it follows that

\begin{equation*}
Q_N D_N = \left\{
\begin{array}{cc}
I - \textrm{Proj}_{\textrm{Ker} \ D_N} & \textrm{for} \ N\ge0 \\
I & \textrm{for} \ N<0
\end{array}\right.
\end{equation*}
and

\begin{equation*}
D_N Q_N = \left\{
\begin{array}{cc}
I & \textrm{for} \ N\ge0 \\
I - \textrm{Proj}_{\textrm{Ker} \ {D_N}^*} & \textrm{for} \ N<0 .
\end{array}\right.
\end{equation*}

From the construction, the kernel of each $T$ operator is the cokernel of the corresponding $A$ operator, which implies
that $Ker(Q_N)=Coker(D_N)$.

Finally all that remains is to show that $Q_N$ is a bounded, and in fact, a compact operator.  
Notice that ${T}^{(n-1)}V^{(n-1)}$ and $-V^{(m)} \overline{T_0}^{(m)}$ are compact operators (in fact Hilbert-Schmidt operators) with norms
that can be estimated as follows:
\begin{equation*}
||T^{(n-1)}V^{(n-1)}||\leq \frac{1}{w_0}\sqrt{C^{(n-1)}C^{(n)}}
\end{equation*}
and similarly 
\begin{equation*}
||V^{(m)} \overline{T_0}^{(m)}||\leq \frac{1}{w_0}\sqrt{C^{(m)}C^{(m+1)}}.
\end{equation*}
Since $C^{(n)}\to 0$ as $n\to\infty$,
it follows from the decomposition that $Q_N$ is compact as a uniform limit of compact operators.
Thus this completes the proof.
\end{proof}

Now we consider the non-commutative cylinder case.   The operator $D_{M,N}$ equals the bilateral operator $D$ with domain

\begin{equation*}
\textrm{dom}(D_{M,N}) = \left\{ a\in \textrm{Dom}(D) : \ r_+(a)\in \textrm{Ran} \ P_N^+ ,\ r_-(a)\in\textrm{Ran} \ P_M^-\right\} .
\end{equation*}

\begin{theo}
The operator $D_{M,N}$ above is an unbounded Fredholm operator with index $ind(D_{M,N}) = M+N+1$. In fact, there is a bounded operator $Q_{M,N}$ such that that $Ker(Q_{M,N})=Coker(D_{M,N})$,
$D_{M,N}Q_{M,N}= I   - Proj_{Coker(D_{M,N})}$,  and   $Q_{M,N}D_{M,N} = I- Proj_{Ker(D_{M,N})}$.
Moreover the parametrix $Q_{M,N}$ is a compact operator.
\end{theo}

\begin{proof}
The proof is analogous to the previous proof, however there are more cases to consider.  This is due to the way we treated
both the disk and the cylinder in complete parallel so far. A different Fourier transform of the Hilbert space could also have been considered
leading to an easier index calculation. However that would have made the corresponding decompositions of $D$ different and more complicated to analyze.

Let 
$a = \sum_{n=0}^\infty U^n f_n(K) + \sum_{n=1}^\infty g_n(K)(U^*)^n$ be in $\textrm{dom}(D_{M,N})$. Then we have
\begin{equation*}
r_\pm (a) = \sum_{m=0}^\infty e^{im\f}f_n({\pm\infty}) + \sum_{n=1}^\infty g_n({\pm\infty}) e^{-in\f}.
\end{equation*}
We need $r_+(a)$ to be in $\textrm{Ran }P^+_N = \underset{n\le N}{\textrm{span}}\{e^{in\f}\}$, 
and for $r_-(a)$ to be in $\textrm{Ran }P^-_M = \underset{-M\le n}{\textrm{span}}\{e^{in\f}\}$, so
one is led to consider the following six cases.
In each case we list the decomposition of the operator $D_{M,N}$ (in the first line), its adjoint ${D_{M,N}}^*$
(in the second line), and the parametrix $Q_{M,N}$ (in the third line). 

{\it \underline{Case 1} : $M+N\ge0$}

{\it Case 1(a) : $N\ge0$, $M>0$}

\begin{equation*}
\left((-\overline{A}^{(m)}W^{(m)})_{m=0}^N ,  (-\overline{A_0}^{(m)}W^{(m)})_{m=N+1}^\infty,
(W^{(n-1)}A^{(n-1)})_{n=1}^M , (W^{(n-1)}A_1^{(n-1)})_{n=M+1}^\infty\right)  
\end{equation*}

\begin{equation*}
\left((-W^{(m)}{A_2}^{(m)})_{m=0}^N ,  (-W^{(m)}{A_1}^{(m)})_{m=N+1}^\infty,
(\overline{A_2}^{(n-1)}W^{(n-1)})_{n=1}^M , (\overline{A_0}^{(n-1)}W^{(n-1)})_{n=M+1}^\infty\right)  
\end{equation*}

\begin{equation*}
\left((-V^{(m)}\overline{T}^{(m)})_{m=0}^N ,  (-V^{(m)}\overline{T_0}^{(m)})_{m=N+1}^\infty,
(T^{(n-1)}V^{(n-1)})_{n=1}^M , (T_1^{(n-1)}V^{(n-1)})_{n=M+1}^\infty\right)  
\end{equation*}

{\it Case 1(b) : $N<0$, $M>0$}

\begin{equation*}
\left((-\overline{A_0}^{(m)}W^{(m)})_{m=0}^\infty,  (W^{(n-1)}A_0^{(n-1)})_{n=1}^{-N-1},
(W^{(n-1)}A^{(n-1)})_{n=-N}^M , (W^{(n-1)}A_1^{(n-1)})_{n=M+1}^\infty\right)  
\end{equation*}

\begin{equation*}
\left((-W^{(m)}{A_1}^{(m)})_{m=0}^\infty,  (\overline{A_1}^{(n-1)}W^{(n-1)})_{n=1}^{-N-1},
(\overline{A_2}^{(n-1)}W^{(n-1)})_{n=-N}^M , (\overline{A_0}^{(n-1)}W^{(n-1)})_{n=M+1}^\infty\right)
\end{equation*}

\begin{equation*}
\left((-V^{(m)}\overline{T_0}^{(m)})_{m=0}^\infty,  (T_0^{(n-1)}V^{(n-1)})_{n=1}^{-N-1},
(T^{(n-1)}V^{(n-1)})_{n=-N}^M , (T_1^{(n-1)}V^{(n-1)})_{n=M+1}^\infty\right)  \end{equation*}
In the formulas above there is no second term when $N=-1$.

{\it Case 1(c) : $M\le0$, $N\ge0$}

\begin{equation*}
\left((-\overline{A_1}^{(m)}W^{(m)})_{m=0}^{-M-1},  (-\overline{A}^{(m)}W^{(m)})_{m=-M}^{N},
(-\overline{A_0}^{(m)}W^{(m)})_{m=N+1}^{\infty}, (W^{(n-1)}A_1^{(n-1)})_{n=1}^\infty\right)  
\end{equation*}

\begin{equation*}
\left((-W^{(m)}{A_0}^{(m)})_{m=0}^{-M-1},  (-W^{(m)}{A_2}^{(m)})_{m=-M}^{N},
(-W^{(m)}{A_1}^{(m)})_{m=N+1}^{\infty}, (\overline{A_0}^{(n-1)}W^{(n-1)})_{n=1}^\infty\right)  
\end{equation*}

\begin{equation*}
\left((-V^{(m)}\overline{T_1}^{(m)})_{m=0}^{-M-1},  (-V^{(m)}\overline{T}^{(m)})_{m=-M}^{N},
(-V^{(m)}\overline{T_0}^{(m)})_{m=N+1}^{\infty}, (T_1^{(n-1)}V^{(n-1)})_{n=1}^\infty\right) 
\end{equation*}
When $M=0$ in the above formulas we simply omit the first term.

{\it \underline{Case 2} : $M+N<0$}

\medskip

{\it Case 2(a) : $N<0$, $M\le0$}

\begin{equation*}
\left((-\overline{A_2}^{(m)}W^{(m)})_{m=0}^{-M-1} ,  (-\overline{A_0}^{(m)}W^{(m)})_{m=-M}^\infty,
(W^{(n-1)}A_2^{(n-1)})_{n=1}^{-N-1} , (W^{(n-1)}A_1^{(n-1)})_{n=-N}^\infty\right)  
\end{equation*}

\begin{equation*}
\left((-W^{(m)}{A}^{(m)})_{m=0}^{-M-1} ,  (-W^{(m)}{A_1}^{(m)})_{m=-M}^\infty,
(\overline{A}^{(n-1)}W^{(n-1)})_{n=1}^{-N-1} , (\overline{A_0}^{(n-1)}W^{(n-1)})_{n=-N}^\infty\right)  
\end{equation*}

\begin{equation*}
\left((-V^{(m)}\overline{T_2}^{(m)})_{m=0}^{-M-1} ,  (-V^{(m)}\overline{T_0}^{(m)})_{m=-M}^\infty,
(T_2^{(n-1)}V^{(n-1)})_{n=1}^{-N-1} , (T_1^{(n-1)}V^{(n-1)})_{n=-N}^\infty\right)  
\end{equation*}

In the formulas above there is no first term when $M=0$.

{\it Case 2(b) : $N<0$, $M>0$}

\begin{equation*}
\left((-\overline{A_0}^{(m)}W^{(m)})_{m=0}^\infty,  (W^{(n-1)}A_0^{(n-1)})_{n=1}^{M},
(W^{(n-1)}A_2^{(n-1)})_{n=M+1}^{-N-1} , (W^{(n-1)}A_1^{(n-1)})_{n=-N}^\infty\right)  
\end{equation*}

\begin{equation*}
\left((-W^{(m)}{A_1}^{(m)})_{m=0}^\infty,  (\overline{A_1}^{(n-1)}W^{(n-1)})_{n=1}^{M},
(\overline{A}^{(n-1)}W^{(n-1)})_{n=M+1}^{-N-1} , (\overline{A_0}^{(n-1)}W^{(n-1)})_{n=-N}^\infty\right)  
\end{equation*}

\begin{equation*}
\left((-\overline{V^{(m)} T_0}^{(m)})_{m=0}^\infty,  (T_0^{(n-1)}V^{(n-1)})_{n=1}^{M},
(T_2^{(n-1)}V^{(n-1)})_{n=M+1}^{-N-1} , (T_1^{(n-1)}V^{(n-1)})_{n=-N}^\infty\right)  
\end{equation*}

{\it Case 2(c) : $N\ge0$, $M<0$}

\begin{equation*}
\left((-\overline{A_1}^{(m)}W^{(m)})_{m=0}^{N-1},  (-\overline{A_2}^{(m)}W^{(m)})_{m=N}^{-M-1},
(-\overline{A_0}^{(m)}W^{(m)})_{m=-M}^{\infty}, (W^{(n-1)}A_1^{(n-1)})_{n=1}^\infty\right) 
\end{equation*}

\begin{equation*}
\left((-W^{(m)}{A_0}^{(m)})_{m=0}^{N-1},  (-W^{(m)}{A}^{(m)})_{m=N}^{-M-1},
(-W^{(m)}{A_1}^{(m)})_{m=-M}^{\infty}, (\overline{A_0}^{(n-1)}W^{(n-1)})_{n=1}^\infty\right) 
\end{equation*}

\begin{equation*}
\left((-V^{(m)}\overline{T_1}^{(m)})_{m=0}^{N-1},  (-V^{(m)}\overline{T_2}^{(m)})_{m=N}^{-M-1},
(-V^{(m)}\overline{T_0}^{(m)})_{m=-M}^{\infty}, (T_1^{(n-1)}V^{(n-1)})_{n=1}^\infty\right) 
\end{equation*}

In the formulas above there is again no first term when $N=0$.

From these representations and from proposition (\ref{prelimsbiker}), one gets the following

\begin{equation*}
\textrm{dim} \ \textrm{Ker}(D_{M,N}) = \left\{
\begin{array}{cc}
M+N+1 & \textrm{for} \ M+N\ge0 \\
0 & \textrm{for} \ M+N<0,
\end{array}\right.
\end{equation*}
and

\begin{equation*}
\textrm{dim} \ \textrm{Ker}({D_{M,N}}^*) = \left\{
\begin{array}{cc}
0 & \textrm{for} \ M+N\ge0 \\
-(M+N+1) & \textrm{for} \ M+N<0.
\end{array}\right.
\end{equation*}
Thus index calculation follows.   
Using the analysis done in section 4, we get the following two relations

\begin{equation*}
Q_{M,N}D_{M,N} = \left\{
\begin{array}{cc}
I - \textrm{Proj}_{\textrm{Ker} D_{M,N}} & \textrm{for} \ M+N \ge 0 \\
I & \textrm{for} \ M+N < 0,
\end{array}\right.
\end{equation*}
and

\begin{equation*}
D_{M,N}Q_{M,N} = \left\{
\begin{array}{cc}
I & \textrm{for} \ M+N \ge 0 \\
I - \textrm{Proj}_{\textrm{Ker} {D_{M,N}}^*} & \textrm{for} \ M+N < 0.
\end{array}\right. 
\end{equation*}

The relation $Ker(Q_{M,N})=Coker(D_{M,N})$ follows from the same property of the parametrix of each component of
$Q_{M,N}$.

The proof that $Q_{M,N}$ is compact is the same as in the unilateral case.
\end{proof}

\end{document}